\documentclass{amsart}
\usepackage{amsthm}
\usepackage{amsmath,amssymb}
\usepackage{mathrsfs}
\usepackage[foot]{amsaddr}
\usepackage[colorlinks=true,linkcolor=blue,citecolor=red]{hyperref}
\usepackage[nameinlink]{cleveref}

\theoremstyle{plain}
\newtheorem{theorem}{Theorem}
\newtheorem*{theorem*}{Theorem}
\newtheorem{lemma}[theorem]{Lemma}
\newtheorem{proposition}[theorem]{Proposition}

\theoremstyle{definition}

\theoremstyle{remark}
\newtheorem{remark}[theorem]{Remark}

\newcommand{\RR}{\mathbb{R}}

\newcommand{\NN}{\mathbb{N}}
\newcommand{\ZZ}{\mathbb{Z}}

\newcommand{\test}{C_c^{\infty}}

\newcommand{\Om}{\Omega}

\newcommand{\I}{\int_}

\allowdisplaybreaks

\begin{document}

	\title[]
	{Bourgain-Brezis-Mironescu formula for $W^{s,p}_q$-spaces in arbitrary domains}
	\author{Kaushik Mohanta}
	\address{Department of Mathematics and Statistics, University of Jyv\'askyl\'a, Finland}
\begin{abstract}
	Under certain restrictions on $s,p,q$, the Triebel-Lizorkin spaces can be viewed as generalised fractional Sobolev spaces $W^{s,p}_q$. In this article, we show that the Bourgain-Brezis-Mironescu formula holds for $W^{s,p}_q$-seminorms in arbitrary domain. This addresses an open question raised by Brazke-Schikorra-Yung in [Bourgain-Brezis-Mironescu convergence via
	Triebel-Lizorkin spaces; Calc. Var. Partial Differential Equations; 2023].
\end{abstract}	
	\maketitle
	
	
\section{Introduction}
Sobolev spaces arise naturally in the study of partial differential equations. They are defined in terms of weak derivatives. For an open set $\Om\subseteq \RR^N$, and $1\leq p<\infty$, the Sobolev space $W^{1,p}(\Om)$ is defined to be $\{f\in L^p(\Om)\ |\ [f]_{W^{1,p}(\Om)}<\infty\}$,
where 
$$
[f]_{W^{1,p}(\Om)}^p:= \I{\Om}|\nabla f(x)|^pdx.
$$
Some closely related spaces are Triebel-Lizorkin spaces $F^s_{p,q}(\Om)$, which are defined to be $\{f|_{\Om}\ |\ f\in F^s_{p,q}(\RR^N) \}$, and are equipped with the norm $[\cdot]_{F^s_{p,q}(\RR^N)}$. We shall not define the norm $[\cdot]_{F^s_{p,q}(\RR^N)}$ as it will not be necessary for the present article. Instead, we refer the reader, for  definition and classical results regarding Triebel-Lizorkin spaces, to \cite{Triebel83}, or more modern references like \cite{Grafa04,RunSic}. 
For $1\leq p,q<\infty$, $\max\{0,\frac{N(q-p)}{pq}\}<s<1$, we have the characterisation (see Theorem~1.2 of \cite{PrSa})
$$
F^s_{p,q}(\RR^N):= \left\{f\in L^{\max\{p,q\}}(\RR^N)\ \Big|\ \|f\|_{L^p(\RR^N)}+[f]_{W^{s,p}_q(\RR^N)}<\infty \right\},
$$
where
\begin{equation}\label{Norm}
[f]_{W^{s,p}_q(\Om)}:= \left(\I{\Om}\left(\I{\Om}\frac{|f(x)-f(y)|^q}{|x-y|^{N+sq}}dy\right)^\frac{p}{q}\right)^\frac{1}{p}.
\end{equation}
In the special case of $p=q$, these spaces are related to the so-called fractional Sobolev spaces $W^{s,p}(\Om)$, defined by
$$
W^{s,p}(\Om):= \left\{f\in L^p(\Om)\ \Big|\ \|f\|_{L^p(\Om)}+[f]_{W^{s,p}(\Om)}<\infty \right\},
$$
where $[f]_{W^{s,p}(\Om)}:=[f]_{W^{s,p}_p(\Om)}$. Of course, when $\Om=\RR^N$, or when $\Om$ is fractional extension domain (see \cite{hhg}), we have $W^{s,p}(\Om)=F^s_{p,p}(\Om)$.\smallskip

Bourgain-Brezis-Mironescu \cite{BBM} showed that for any smooth and bounded domain $\Om$, $1\leq p<\infty$, and any $f\in W^{1,p}(\Om)$, 
$$
\lim_{s\to1-}(1-s)[f]_{W^{s,p}(\Om)}^p
= K\|\nabla f\|_{L^p(\Om)}^p.
$$ 
Conversely, for any $f\in L^p(\Om)$, if we have
$$
\lim_{s\to1-}(1-s)[f]_{W^{s,p}(\Om)}^p<\infty,
$$
then $f\in W^{1,p}(\Om)$ if $p>1$ and $f\in BV(\Om)$ if $p=1$. Later D\'avila \cite{Dav} extended this result and proved that for any $f\in BV(\Om)$, $\lim_{s\to1-}(1-s)[f]_{W^{s,1}(\Om)}^p=K|\nabla f|(\Om)$. 
This is commonly known as the Bourgain-Brezis-Mironescu formula (BBM formula for short). The subject was further developed in \cite{BoBrMi02,BoNg,Brezis02,Nguyen06,Nguyen08,Ponce04}. The BBM formula has been generalised to Orlicz and generalised Orlicz setup in \cite{AlCiPi20,AlCiPi,AlCiPi21,FeSa,FeHa,YaYaYu}, to variable exponent setup in \cite{FeSq,HaRi},
to magnetic Sobolev spaces \cite{FeSa21,NgPiSq,NgPiSq20,PiSqVe,SqVo}, to anisotropic setup \cite{GuHu,Ludwig14,NgSq}, to Riemannian manifolds in \cite{KrMo}, to metric spaces in \cite{DiSq,Gorny22}, to Banach function spaces in \cite{DGPYYZ,ZhYaYu}. Similar studies are possible in the context of Besov spaces also \cite{KoLe,Triebel11}. More recent developments on this topic can be found in \cite{BrNg,BrVaYu,BrVaJe,BrSeVa}. For further reading, we refer the reader to \cite{ArBo,BaLi,Bojarski,BrGoDa,BuGaTr,CuLaLu,DoMi,Foghem,Foghem-thesis,FoKaVo,Hajlas03,Milman05,PoSp}.\smallskip

Our first interest lies in works regarding the domain. That is, how far the smooth-bounded condition, on the domain, can be relaxed. In this direction, we mention three works. The first is due to Lioni-Spector \cite{LeSp,LeSp-corr}. They showed that for an arbitrary domain $\Om$, and $f\in L^p(\Om)$,
$$
\lim_{\lambda\to 0+}\lim_{s\to 1-}(1-s)\I{\Om_\lambda}\I{\Om_\lambda}\frac{|f(x)-f(y)|^p}{|x-y|^{N+sp}}dydx=K[f]_{W^{1,p}(\Om)}^p.
$$
where
\begin{equation}\label{omegalambda}	
\Om_\lambda:=\{x\in\Om\ |\ \mbox{dist}(x,\partial\Om)>\lambda\}\cap B(0,\lambda^{-1}).
\end{equation}
The second result is due to the author with Bal-Roy \cite{BaMoRo}, where it has been shown that the BBM-formula \cite{BBM}, holds if we take $\Om$ to be a $W^{1,p}$-extension domain. The third result is due to Drelichman-Duran \cite{DrDu}. They showed that for $1<p<\infty$, and an arbitrary bounded domain $\Om$, and any $\tau\in(0,1)$, we have
$$
\lim_{s\to 1-}(1-s)\I{\Om}\I{B(x,\tau \mbox{dist}(x,\partial \Om))}\frac{|f(x)-f(y)|^p}{|x-y|^{N+sp}}dydx=K[f]_{W^{1,p}(\Om)}^p.
$$
\smallskip

The second direction of work regarding the BBM-formula, that we are interested in, is its extension for Triebel-Lizorkin spaces. The first work in this direction was done by Brazke-Schikorra-Yung \cite{BrScYu}. They explained via examining thoroughly various constants of embeddings that although $F^s_{p,p}=W^{s,p}$, when $s\in(0,1)$ and $F^s_{p,2}=W^{1,p}$ it makes sense for the scaled $W^{s,p}$ seminorm to converge to $W^{1,p}$-seminorm, even when $p\neq 2$. They posed the open problem (see \cite[Question~1.12]{BrScYu}) about the asymptotic constant in the identification of the $\|\cdot\|_{W^{s,p}_q}\approx \|\cdot\|_{F^s_{p,q}}$. \smallskip
	
The current article addresses this question by showing the asymptotic behaviour (as $s\to 1-$) of $W^{s,p}_q$-seminorms. Similar studies has been done, when $1<q<p<\infty$, in \cite{DGPYYZ} for $\RR^N$, and in \cite[Theorem~6.1]{ZhYaYu} for a special class of bounded extension domains (called $(\varepsilon,\infty)$-domains).\smallskip

We concentrate our focus on the following seminorm (for some $\tau\in (0,1)$)
\begin{equation}\label{Norm-int}
	[f]_{\tilde{W}^{s,p}_q(\Om)}:= \left(\I{\Om}\left(\I{B(x,\tau \mbox{dist}(x,\partial \Om))}\frac{|f(x)-f(y)|^q}{|x-y|^{N+sq}}dy\right)^\frac{p}{q}\right)^\frac{1}{p},
\end{equation}
as one can then extend the above question for arbitrary bounded domains, motivated by \cite{DrDu}. We go one step further and show that the boundedness of the domain is not necessary. Our main results are the following:\smallskip

\begin{theorem}\label{Main1*}
	Let $\Om\subset \RR^N$ be any open set $\tau\in(0,1)$. Assume one of the following conditions
	\begin{enumerate}
		\item $1\leq q\leq p<\infty$,
		\item $1<p<q<\infty$ with $p\leq N$ and $q <\frac{Np}{N-p}$, 
		\item $N<p<q<\infty$.
	\end{enumerate} Then there is a constant $K=K(N,p,q)>0$ such that for any $f\in W^{1,p}(\Omega)$, we have,
	\begin{equation}\label{bbm*}
		\lim_{s\to1-} (1-s)^\frac{p}{q}\I{\Om}\left(\I{B(x,\tau\mbox{dist}(x,\partial\Om))}\frac{|f(x)-f(y)|^q}{|x-y|^{N+sq}}dy\right)^\frac{p}{q}dx
		= K\I{\Om}|\nabla f(x)|^pdx.
	\end{equation}
\end{theorem}
\smallskip

\begin{theorem}\label{Main2*}
	Let $\Om\subset \RR^N$ be an open set, $\tau\in(0,1)$ and $1\leq p,q<\infty$. If $f\in L^p(\Om)\cap L^q(\Om)$ is such that 
	$$
	L^*_{p,q}(f):=\lim_{s\to1-} \I{\Om}\left((1-s)\I{B(x,\tau\mbox{dist}(x,\partial\Om))}\frac{|f(x)-f(y)|^q}{|x-y|^{N+sq}}dy\right)^\frac{p}{q}dx<\infty,
	$$  
	then $f\in W^{1,p}(\Om)$ when $p>1$, and $f\in BV(\Om)$ when $p=1$.
\end{theorem}
\smallskip

\begin{remark}
	Before proceeding further, let us discuss some difficulties that arise here, and strategies for overcoming them. The proof of the main results roughly follow the outline of \cite{BBM}. However, there are certain obstacles to that path. The first obstacle arises when we want to apply dominated convergence theorem to interchange limit and integral. Similar difficulty was faced and overcome in \cite{DrDu}, but we had to take a different route (see \cref{limit-interchange}) for this purpose. The introduction of the second exponent $q$ forces us to deviate from the usual route again; The case $q\leq p$ is rather easy to handle, for the case $p<q$, a careful use of Sobolev embedding is needed. To take into account the case where the domain $\Om$ is unbounded, we need to restrict the seminorm further and define some new fractional Sobolev spaces (see \cref{new-space}) and prove a version of the main result \cref{Main1*} in that context (see \cref{Main1}), and then finally derive the proof of the main results from there.
\end{remark}
\smallskip

The article is organised as follows: In \cref{Preli}, we list some preliminary results, already known in literature, which shall be useful for the proof of our main results. In \cref{Functional}, we introduce a variant of fractional Sobolev spaces and prove some relevant embedding results. In \cref{Restricted-proof} we prove the main results in the context of these new spaces (see \cref{Main1,Main2}). Finally in \cref{Proof}, we prove \cref{Main1*,Main2*}. 
\bigskip

\section{Preliminary Results}\label{Preli}
For the sake of completeness, we first state the well-known Sobolev inequality (also known as $(q,p)$-Poincar\'e inequality): 
\begin{lemma}\label{sobolev}
Let $1\leq p,q\leq\infty$, $\tau\in(0,1)$, and one of the following hold
\begin{enumerate}
\item $p<N$, and $q \leq\frac{Np}{N-p}$,
\item $p=N$, and $q <\infty$
\item $p>N$.
\end{enumerate}
Then there is a constants $C=C(p,q,N)>0$ such that the following holds for any $f\in W^{1,p}(B(0,\frac{t}{\tau}))$:
$$
\frac{1}{t^N}\I{B(0,t)}|f(y)|^qdy
\leq C(N,p,q) t^q \left(\frac{1}{t^N}\I{B(0,t)}|\nabla f (y)|^p dy\right)^\frac{q}{p}
+ C(N,p,q) \left(\frac{1}{t^N}\I{B(0,t)}|f(y)|^p dy\right)^\frac{q}{p}.
$$
\end{lemma}\smallskip

The following lemma was established in \cite{Calderon61} for $1<p<\infty$; the $p\geq 1$ case can be found in \cite[Chapter-VI, Theorems 5 and 5']{Stein70}.
\begin{lemma}\label{stein}
	Let $\Om\subseteq\RR^N$ be an open set with Lipschitz boundary and $1\leq p<\infty$. Then for any $f\in W^{1,p}(\Om)$ there is some $\tilde{f}\in W^{1,p}(\RR^N)$ such that $\tilde{f}|_{\Om}=f$ and  for some constant $C=C(N,\Om,p)$, 
	$$
	\|\tilde{f}\|_{W^{1,p}(\RR^N)}\leq C\|f\|_{W^{1,p}(\Om)}.
	$$
\end{lemma}\smallskip

The following result can be found in Proposition~9.3 and Remark~6 of \cite{Bre11}.
\begin{lemma}\label{brezis-lemma}
	Let $\Om\subseteq \RR^N$ be open, $1\leq p<\infty$, $\frac{1}{p}+\frac{1}{p'}=1$, and $f\in L^p(\Om)$. Assume that there is constant $C>0$ such that for any $\varphi\in \test(\Om)$
	$$
	\left|\I{\Om}f(x)\frac{\partial \varphi(x)}{\partial x_i}dx\right|
	\leq C\|\varphi\|_{L^{p'}(\Om)} 
	\quad \mbox{for } i=1,2,\cdots, N.
	$$
	Then $f\in W^{1,p}(\Om)$ when $1<p$, and $f\in BV(\Om)$ when $p=1$.
\end{lemma}\smallskip

Next we list a special case of Proposition~2/(ii) of \cite[Chapter~2.3.3]{Triebel83}, combined with the fact that $W^{1,p}(\RR^N)=F^1_{p,2}(\RR^N)$.
\begin{lemma}\label{triebel}
Let $1\leq p,q<\infty$, $s\in(0,1)$. Then 
$$
W^{1,p}(\RR^N)\subseteq F^s_{p,q}(\RR^N) = W^{s,p}_q(\RR^N).
$$ 
\end{lemma}\smallskip

The following result is taken from Lemma~8 of \cite{BaMoRo}.
\begin{lemma}\label{exhaustion}
Let $\Om\subseteq \RR^N$ be open and $\lambda > 0$ be sufficiently small. Then there is a bounded open set $\Om_\lambda^*$ with smooth boundary such that $\Om_\lambda \subseteq \Om_\lambda^* \subseteq \Om$, where $\Om_\lambda$ is as in \eqref{omegalambda}.
\end{lemma}\smallskip

The next result can be found in Theorem~2.1 of \cite{Takashi}. It will play a crucial in this article.
\begin{lemma}\label{limit-interchange}
Let $\Om\subseteq \RR^N$, $\{\Om_i\}_{i\in\NN}$ be such that $\Om=\cup_i\Om_i$, $F_n,F\in L^1(\Om)$ for $n\in \NN$ be such that for a.e. $x\in\Om$, $F_n(x)\to F(x)$ as $n\to\infty$. Assume that
\begin{enumerate}
	\item $$
	\quad \quad \limsup_{n\to\infty}\sup_{x\in\Om_i}\left| F_n(x)-F(x) \right|<\infty \quad \mbox{for all } i\in \NN,
	$$
	\item $$
	\quad \quad \liminf_{i\to\infty}\limsup_{n\to\infty}\I{\Om\setminus\Om_i} F_n(x)dx=0.
	$$	
\end{enumerate}
Then 
$$
\lim_{n\to\infty}\I{\Om}F_n(x)dx
=\I{\Om}F(x).
$$
\end{lemma}
\bigskip

\section{Fractional Sobolev space with restricted internal distance}\label{Functional}

Fix $R>0$ and $\tau\in(0,1)$ once and for all. Denote $\delta_{x,R,\tau}= \min\{R,\tau\mbox{dist}(x,\partial\Om)\}$. We shall often drop the $R$ and $\tau$ in the above notation and write $\delta_x$ to denote $\delta_{x,R,\tau}$.
\begin{remark}
	If the function $x\mapsto \mbox{dist}(x,\partial\Om)$ is bounded in $\Om$, we can choose $R>0$ large enough, so that $\delta_x=\tau\mbox{dist}(x,\partial\Om)$. Then the particular case $p=q$ of \cref{Main1,Main2} are similar to the results proved in \cite{DrDu}, but here $\Om$ need not be a bounded domain; for example, it can be a cylindrical domain or any open subset of $\RR^N\setminus\ZZ^N$.
\end{remark}
Define, for any open set $\Om\subseteq \RR^N$, $1\leq p,q<\infty$, $0<s<1$, $\hat{W}^{s,p}_q(\Omega):=\{f\in L^p(\Om)\ |\ [f]_{\hat{W}^{s,p}_q(\Omega)}<\infty\}$ where
\begin{equation}\label{new-space}
[f]_{\hat{W}^{s,p}_q(\Omega)}^p:=\I{x\in\Om}\left(\I{y\in B(x,\delta_x)}\frac{|f(x)-f(y)|^q}{|x-y|^{N+sq}}dy\right)^\frac{p}{q}dx.
\end{equation}
We shall need some embedding results for these new fractional Sobolev spaces for our purpose. As expected, the case $q\leq p$ and $p<q$ are treated separately.
\begin{lemma}\label{embedding-p>q}
	Let $\Om\subseteq \RR^N$ be open and $1\leq q\leq p <\infty$. Assume either $D=\tilde{D}= \Omega$ or, for some $\frac{1}{2R}>\alpha>0$, $D= \{x\in\Omega\ |\ \mbox{dist}(x,\partial\Om)<\alpha \}$ with $\tilde{D}=\{x\in\Omega\ |\ \mbox{dist}(x,\partial\Om)<2\alpha \mbox{ or } |x|> \frac{1}{2\alpha} \}$. Then there is a constant $C=C(p,q,R,\Om,N)$ such that for any $f$ in $W^{1,p}(\Omega)$,
	\begin{equation}\label{eq-embedding}
		(1-s)^\frac{p}{q}\I{x\in D}\left(\I{y\in B(x,\delta_x)} \frac{| f(x)-f(y) |^q}{| x-y |^{N+sq}} dy\right)^\frac{p}{q}dx \leq C [f]_{W^{1,p}(\tilde{D})}^p.
	\end{equation}
\end{lemma}
\begin{proof}
	We have
	\begin{align*}
		(1-s)^\frac{p}{q}\I{x\in D}&\left(\I{y\in B(x,\delta_x)} \frac{| f(x)-f(y) |^q}{| x-y |^{N+sq}} dy\right)^\frac{p}{q}dx\\
		=& (1-s)^\frac{p}{q}\I{x\in D} \left(\I{h\in B(0,\delta_x)} \frac{| f(x+h)-f(x) |^q}{| h |^{N+sq}}  dh\right)^\frac{p}{q}dx\\
		=& (1-s)^\frac{p}{q}\I{x\in D} \left(\I{h\in B(0,\delta_x)} \frac{| f(x+h)-f(x) |^q}{| h |^p} \frac{dh}{|h|^{N+sq-q}}\right)^\frac{p}{q}dx\\
		\leq& (1-s)^\frac{p}{q}\I{x\in D}\left(\I{h\in B(0,\delta_x)} \I{0}^1|\nabla f(x+th)|^qdt \frac{dh}{|h|^{N+sq-q}}\right)^\frac{p}{q}dx.
	\end{align*}
	The last inequality follows from the absolute continuity on lines of the $W^{1,p}$-functions. We now have, after a change of variable $y=x+th$, (using that $B(x,t\delta_x) \subset B(x,\delta_x)$)
	\begin{align*}
		(1-s)^\frac{p}{q}\I{x\in D}&\left(\I{y\in B(x,\delta_x)} \frac{| f(x)-f(y) |^q}{| x-y |^{N+sq}} dy\right)^\frac{p}{q}dx\\
		\leq& (1-s)^\frac{p}{q}\I{x\in D}\left(\I{y\in B(x,\delta_x)} \I{0}^1|\nabla f(y)|^qt^{sq-q}dt \frac{dy}{|x-y|^{N+sq-q}}\right)^\frac{p}{q}dx\\
		=& \frac{(1-s)^\frac{p}{q}}{(sq-q+1)^\frac{p}{q}}\I{x\in D}\left(\I{y\in B(x,\delta_x)} |\nabla f(y)|^q \frac{dy}{|x-y|^{N+sq-q}}\right)^\frac{p}{q}dx.
	\end{align*}
	Note that in the above inequality, $\nabla f$ is required to be defined only inside $\tilde{D}$. So, we shall take a $0$-extension of $\nabla f$ outside $\tilde{D}$. Since we have $\frac{p}{q}\geq 1$, we can use Young's convolution inequality to get
	\begin{align*}
		(1-s)^\frac{p}{q}\I{x\in D}&\left(\I{y\in B(x,\delta_x)} \frac{| f(x)-f(y) |^q}{| x-y |^{N+sq}} dy\right)^\frac{p}{q}dx\\
		\leq&  \frac{(1-s)^\frac{p}{q}}{(sq-q+1)^\frac{p}{q}}\I{x\in D}\left(\I{y\in B(x,R)} |\nabla f(y)|^q \frac{dy}{|x-y|^{N+sq-q}}\right)^\frac{p}{q}dx\\
		\leq& \frac{(1-s)^\frac{p}{q}}{(sq-q+1)^\frac{p}{q}}\I{x\in \tilde{D}}|\nabla f(x)|^pdx\left(\I{x\in B(0,R)} \frac{dx}{|x|^{N+sq-q}}\right)^\frac{p}{q}\\
		=& \frac{R^{p-sp}}{q^\frac{p}{q}}\I{x\in \tilde{D}}|\nabla f(x)|^pdx.
	\end{align*}
\end{proof}

\begin{lemma}\label{embedding-p<q}
	Let $\Omega\subseteq\mathbb{R}^N$ be open, $1<p<q<\infty$, and one of the following hold
	\begin{enumerate}
		\item $p<N$, and $q <\frac{Np}{N-p}$,
		\item $N\leq p$.
	\end{enumerate}
Assume either $D=\tilde{D}= \Omega$ or, for some $\frac{1}{2R}>\alpha>0$, $D= \{x\in\Omega\ |\ \mbox{dist}(x,\partial\Om)<\alpha \}$ and $\tilde{D}=\{x\in\Omega\ |\ \mbox{dist}(x,\partial\Om)<2\alpha\mbox{ or } |x|> \frac{1}{2\alpha}\}$.
Then there is a constant $C=C(N,R,p,q)>0$ such that for any $f\in W^{1,p}(\Omega)$, \cref{eq-embedding} holds.
\end{lemma}
\begin{proof}
	Note that
\begin{equation*}
	\I{t=|h|}^{\delta_x}\frac{dt}{t^{N+sq+1}}
	=\frac{1}{N+sq}\left(\frac{1}{|h|^{N+sq}}-\frac{1}{\delta_x^{N+sq}}\right),
\end{equation*}
which gives
\begin{equation*}
	\frac{1}{|h|^{N+sq}}
	= (N+sq)\I{t=|h|}^{\delta_x}\frac{dt}{t^{N+sq+1}}+\delta_x^{-N-sq}.
\end{equation*}
So, we can write
\begin{align}\label{eq-embd-p<q-1}
	C(p,q) \I{x\in D}&\left((1-s)\I{h\in B(0,\delta_x)} \frac{| f(x+h)-f(x) |^q}{| h |^{N+sq}} dh\right)^\frac{p}{q}dx\\
	\nonumber\leq& \I{x\in D}\left((N+sq)(1-s)\I{h\in B(0,\delta_x)} \I{t=|h|}^{\delta_x} \frac{| f(x+h)-f(x) |^q}{t^{N+sq+1}}dt dh\right)^\frac{p}{q}dx\\
	\nonumber &+ \I{x\in D}\left((1-s)\I{h\in B(0,\delta_x)} \frac{| f(x+h)-f(x) |^q}{\delta_x^{N+sq}} dh\right)^\frac{p}{q}dx\\
	\nonumber=& I_1+I_2.
\end{align}
According to either $p<N$ or $p\geq N$, fix $\beta\in(0,1)$, depending on $p,q,N$, such that $(q,\beta p)$-type Poincar\'e inequality (\cref{sobolev}) is satisfied. We use this to estimate $I_1$ below. First, we change the order of integration between $t$ and $h$, then apply \cref{sobolev}. \begin{align}\label{eq2}
	I_1=& \I{x\in D}\left((N+sq)(1-s)\I{t=0}^{\delta_x} \frac{1}{t^{N+sq+1}} \I{h\in B(0,t)} | f(x+h)-f(x) |^q dhdt \right)^\frac{p}{q}dx\\
\nonumber	\leq& \I{x\in D}\left((N+sq)(1-s)\I{t=0}^{\delta_x} \frac{1}{t^{sq+1-q}} \left(\frac{1}{t^N}\I{h\in B(x,t)} |\nabla f(h)|^{\beta p}dh\right)^\frac{q}{\beta p}dt \right)^\frac{p}{q}dx\\
\nonumber	&+ \I{x\in D}\left((N+sq)(1-s)\I{t=0}^{\delta_x} \frac{1}{t^{sq+1}} \left(\frac{1}{t^N}\I{h\in B(0,t)} | f(x+h)-f(x) |^{\beta p} dh \right)^\frac{q}{\beta p}dt \right)^\frac{p}{q}dx\\
\nonumber	=& I_{1,1}+I_{1,2}.
\end{align}
As in the proof of the previous lemma, we shall take a $0$-extension of $\nabla f$ outside $\tilde{D}$.
Now using Hardy-Littelewood maximal inequality, we get
\begin{align}\label{eq3}
I_{1,1}
\leq&  \I{x\in D}(M |\nabla f(x)|^{\beta p})^\frac{1}{\beta}\left((N+sq)(1-s)\I{t=0}^{\delta_x} \frac{1}{t^{sq+1-q}} dt \right)^\frac{p}{q}dx\\
\nonumber=& C(N,p,q,R)\I{x\in \RR^N}(M |\nabla f(x)|^{\beta p})^\frac{1}{\beta}dx\\
\nonumber\leq& C(N,p,q,R)\I{x\in \tilde{D}}|\nabla f(x)|^pdx.
\end{align}
Again,
\begin{align}\label{eq4}
I_{1,2}
\leq& \I{x\in D}\left((N+sq)(1-s)\I{t=0}^{\delta_x} \frac{1}{t^{sq+1-q}} \left(\frac{1}{t^N}\I{h\in B(0,t)} \frac{| f(x+h)-f(x) |^{\beta p}}{|h|^{\beta p}} dh \right)^\frac{q}{\beta p}dt \right)^\frac{p}{q}dx\\
\nonumber\leq& \I{x\in D}\left((N+sq)(1-s)\I{t=0}^{\delta_x} \frac{1}{t^{sq+1-q}} \left(\I{r=0}^1\frac{1}{t^N}\I{h\in B(0,t)} |\nabla f(x+rh)|^{\beta p} dhdr \right)^\frac{q}{\beta p}dt \right)^\frac{p}{q}dx\\
\nonumber=& \I{x\in D}\left((N+sq)(1-s)\I{t=0}^{\delta_x} \frac{1}{t^{sq+1-q}} \left(\I{r=0}^1\frac{1}{(rt)^N}\I{h\in B(x,rt)} |\nabla f(h)|^{\beta p} dhdr \right)^\frac{q}{\beta p}dt \right)^\frac{p}{q}dx\\
\nonumber\leq& \I{x\in \RR^N}(M|\nabla f(x)|^{\beta p})^\frac{1}{\beta}\left((N+sq)(1-s)\I{t=0}^{R} \frac{1}{t^{sq+1-q}} \left(\I{r=0}^1dr \right)^\frac{q}{\beta p}dt \right)^\frac{p}{q}dx\\
\nonumber\leq& C(p,q,N,R) \I{x\in \tilde{D}}|\nabla f(x)|^{p}dx.
\end{align}
Combining \cref{eq3,eq4}, we get 

\begin{equation}\label{eq5}
	I_1 \leq C(p,q,N,R) \|\nabla f\|_{L^p(\tilde{D})}^p.
\end{equation}
Again, we can estimate $I_2$, in similar way as above with $\delta_x$ in place of $t$. We have a better estimate this time. Also, we can apply $(q,p)$-Poincar\'e inequality this time.
\begin{align}\label{eq6}
	I_2
	=& \I{x\in D}\left((1-s)\I{h\in B(0,\delta_x)} \frac{| f(x+h)-f(x) |^q}{\delta_x^{N+sq}} dh\right)^\frac{p}{q}dx\\
\nonumber	\leq& \I{x\in D}\left(\frac{(1-s)}{\delta_x^{sq}}\left(\delta_x^{p-N}\I{h\in B(0,\delta_x)} |\nabla f(x+h)|^p dh
	+ \delta_x^{-N}\I{h\in B(0,\delta_x)} | f(x+h)-f(x) |^p dh\right)^\frac{q}{p}\right)^\frac{p}{q}dx\\
\nonumber	\leq& \I{x\in D}\frac{(1-s)^\frac{p}{q}}{\delta_x^{sp-p}}\left(\left(\delta_x^{-N}\I{h\in B(0,\delta_x)} |\nabla f(x+h)|^p dh
	+ \delta_x^{-N}\I{h\in B(0,\delta_x)} \frac{| f(x+h)-f(x) |^p}{|h|^p} dh\right)^\frac{q}{p}\right)^\frac{p}{q}dx\\
\nonumber	\leq& \I{x\in D}\frac{(1-s)^\frac{p}{q}}{\delta_x^{sp-p}}\left(\delta_x^{-N}\I{h\in B(0,\delta_x)} |\nabla f(x+h)|^p dh
	+ \I{r=0}^1\frac{1}{(r\delta_x)^N}\I{h\in B(0,r\delta_x)} |\nabla f(x+h)|^p dhdr\right)dx\\
\nonumber	\leq& (1-s)^\frac{p}{q}\I{x\in D}\delta_x^{p-sp}|\nabla f(x)|^pdx.
\end{align}

Combing \cref{eq5,eq6,eq-embd-p<q-1},
\begin{equation*}
	\I{x\in D}\left((1-s)\I{h\in B(0,\delta_x)} \frac{| f(x+h)-f(x) |^q}{| h |^{N+sq}} dh\right)^\frac{p}{q}dx
	\leq C(N,p,q,R) \|\nabla f\|_{L^p(\tilde{D})}^p.
\end{equation*}
This proves the lemma.
\end{proof}

\bigskip

\section{BBM formula for \protect{$\hat{W}^{s,p}_q$}-seminorms}\label{Restricted-proof}

First, we state the following result whose proof can be found in the proof of Theorem~ of \cite{BBM} as the quantity $\delta_x$ is bounded by $R$.
\begin{lemma}\label{lm-pointwise}
	Let $\Om\subset \RR^N$ be any open set, $1\leq q<\infty$, $0<s<1$. Then for any $f\in C^2(\Omega)$, we have for all $x\in\Om$,
	\begin{equation}\label{pointwise}
		\lim_{s\to1-} (1-s)\I{B(x,\delta_x)}\frac{|f(x)-f(y)|^q}{|x-y|^{N+sq}}dy
		= K |\nabla f(x)|^q.
	\end{equation}
\end{lemma}

We now prove the following BBM-type results which are closely related with \cref{Main1*,Main2*}.

\begin{theorem}\label{Main1}
	Let $\Om\subset \RR^N$ be any open set. Assume one of the following conditions
	\begin{enumerate}
		\item $1\leq q\leq p<\infty$,
		\item $1<p<q<\infty$ with $p\leq N$ and $q <\frac{Np}{N-p}$, 
		\item $N<p<q<\infty$.
	\end{enumerate} Then there is a constant $K=K(N,p,q)>0$ such that for any $f\in W^{1,p}(\Omega)$, we have for all $x\in\Om$,
	\begin{equation}\label{bbm}
		\lim_{s\to1-} (1-s)^\frac{p}{q}\I{\Om}\left(\I{B(x,\delta_x)}\frac{|f(x)-f(y)|^q}{|x-y|^{N+sq}}dy\right)^\frac{p}{q}dx
		= K\I{\Om}|\nabla f(x)|^pdx.
	\end{equation}
\end{theorem}

\begin{proof}[\textbf{Proof of \Cref{Main1}}]
\textbf{Step-1:} We show that it is enough to prove \cref{bbm} for $f\in W^{1,p}(\Om)\cap C^2(\Om)$.\\

 Let $f \in W^{1,p}(\Omega)$ and $\varepsilon>0$ be fixed. Since $C^2(\Om)\cap W^{1,p}(\Omega)$ is dense in $W^{1,p}(\Omega)$, there exists $g \in C^2(\Om)\cap W^{1,p}(\Omega)$ such that
\begin{equation}\label{1}
	[f-g]_{W^{1,p}(\Omega)} < \varepsilon.
\end{equation}
Since we have assumed that \cref{bbm} holds for functions in $C^2(\Om)\cap W^{1,p}(\Omega)$, for $s>\frac{1}{2}$, we have
\begin{equation}\label{2}
	\left| (1-s)^\frac{1}{q}[g]_{s,p,q,\Omega,R}-K^\frac{1}{p}[g]_{W^{1,p}(\Omega)} \right| <\varepsilon.
\end{equation}
Using triangle inequality, and then \cref{2} followed by \cref{1} and either \cref{embedding-p>q} or \cref{embedding-p<q}, we get 
\begin{align*}
	\left| (1-s)^\frac{1}{q}[f]_{s,p,q,\Om,R}-K^\frac{1}{p}[f]_{W^{1,p}(\Omega)} \right| 
	\leq&  (1-s)^\frac{1}{q}\left| [f]_{s,p,q,\Om,R} - [g]_{s,p,q,\Omega,R} \right|\\
	&+ \left| (1-s)^\frac{1}{q}[g]_{s,p,\Omega,R}-K^\frac{1}{q}[g]_{W^{1,p}(\Omega)} \right| + K^\frac{1}{q}\left| [g]_{W^{1,p}(\Omega)}-[f]_{W^{1,p}(\Omega)}\right|\\
	\leq& (1-s)^\frac{1}{q}[f-g]_{s,p,q,\Omega,R} + \varepsilon+ K^\frac{1}{p}[f-g]_{W^{1,p}(\Omega)}\\
	\leq& [f-g]_{W^{1,p}(\Om)} 
	+\varepsilon+ K^\frac{1}{p}[f-g]_{W^{1,p}(\Omega)}\\
	\leq& C(K,q) \varepsilon.
\end{align*}
The proof of step-1 follows.\\ 

\noindent\textbf{Step-2:} \\

In view of the previous step, it is now enough to assume that $f\in C^2(\Om)\cap W^{1,p}(\Omega)$ and prove \cref{bbm}. Let us take an arbitrary sequence $s_n\in(0,1)$ such that $s_n\to1-$ as $n\to\infty$.
Set 
$$F_n(x):=\left((1-s_n)\I{B(0,\delta_x)}\frac{|f(x+h)-f(x)|^p}{|h|^{N+s_np}}dh\right)^\frac{p}{q},
$$
and
$$
F(x):=K|\nabla f(x)|^p.
$$
Also note that, \cref{lm-pointwise} implies that $F_n\to F$ pointwise a.e in $\Om$.
To complete the proof, it is enough to show that 
$$
\lim_{n\to\infty}\I{\Om}F_n(x)dx=\I{\Om}F(x)dx.
$$
We shall apply \cref{limit-interchange} on $F_n$ to show that the interchange of limit and integral is valid.

For any $i\in \NN$, consider the sets $\Om_i:=\{x\in\Om\ |\ \mbox{dist}(x,\partial \Om)>\frac{1}{i}\}\cap B(0,i)$.
We need to verify the the hypotheses of \cref{limit-interchange}. First, note that for $x\in\Om_i$, $h\in B(0,\delta_x)$, $t\in(0,1)$, we have 
$$
\mbox{dist}(x+th,\partial\Om)
> \mbox{dist}(x,\partial\Om) -|h|
\geq (1-\tau)\mbox{dist}(x,\partial\Om).
$$
Thus $x+th\in \Om_{i^2}$ for $i>\frac{1}{(1-\tau)}$. Thus, we have using triangle inequality and then mean value inequality,
\begin{align*}
	|F_{n}(x)-F(x)|
	&=\left| \left((1-s_n)\I{B(0,\delta_x)}\frac{|f(x+h)-f(x)|^q}{|h|^{N+s_nq}}dh\right)^\frac{p}{q}-K|\nabla f(x)|^p \right|\\
	&\leq \sup_{y\in\Om_{i^2}}|\nabla f(y)|^p\left((1-s_n)\I{B(0,R)}\frac{dh}{|h|^{N+s_nq-q}}\right)^\frac{p}{q} +K|\nabla f(x)|^p\\
	&\leq C(N,p,q,R)\sup_{y\in\Om_{i^2}}|\nabla f(y)|^p.
\end{align*}
Since $f$ is continuous in the closure of the bounded open set $\Om_{i^2}$, we have the hypothesis (1) of \cref{limit-interchange} satisfied for sufficiently large $i\in\NN$.

Note that, to show that hypothesis (2) of \cref{limit-interchange} is satisfied, it is enough to show that $$\lim_{i\to\infty}\lim_{n\to\infty}\I{\Om\setminus\Om_{2i}} F_n(x)dx=0.$$ We start with an arbitrary $x\in\Om\setminus\Om_{2i}$, $h\in B(0,\delta_x)$ and $t\in (0,1)$. There can be two cases:\\
\textbf{Case:1 } $\delta_x=\mbox{dist}(x,\partial\Om)<\frac{1}{2i}$.\\
We have $\mbox{dist}(x+th,\partial\Om)\leq |x+th-x|+\mbox{dist}(x,\partial\Om)<\frac{\tau}{2i}+\frac{1}{2i}<\frac{1}{i}$. Thus $x+th\in \Om\setminus \Om_i$.\\
\textbf{Case:2 } 
$|x|>2i$ and $\frac{1}{2i}<\delta_x=R<\mbox{dist}(x,\partial\Om)$. Moreover, we can assume $R<i$ without loss of generality.\\
We have $|x+th|\geq |x|-\tau R \geq 2i-\tau R\geq i$. Thus $x+th\in\Om\setminus B(0,i)\subseteq \Om\setminus\Om_i$.
Hence we always have
\begin{equation}\label{4}
x+th\in\Om\setminus\Om_i \mbox{ whenever }	x\in\Om\setminus\Om_{2i}.
\end{equation}
From \cref{4}, we get
\begin{equation*}
\lim_{i\to\infty}\lim_{n\to\infty}\I{\Om\setminus\Om_{2i}} F_n(x)dx
=\lim_{i\to\infty}\lim_{n\to\infty} \I{x\in\Om\setminus\Om_{2i}}\left((1-s_n)\I{y\in B(x,\delta_x)} \frac{| f(x)-f(y) |^q}{| x-y |^{N+s_nq}} dy\right)^\frac{p}{q}dx.
\end{equation*}
Now we apply \cref{embedding-p>q} or \cref{embedding-p<q} with $D=\Om\setminus\Om_{2i}$ (so that $\tilde{D}=\Om\setminus\Om_{i}$) to get
\begin{equation*}
\lim_{i\to\infty}\lim_{n\to\infty}\I{\Om\setminus\Om_{2i}} F_n(x)dx
\leq \lim_{i\to\infty}\lim_{n\to\infty} C(p,q,R,N)[f]_{W^{1,p}(\Om\setminus\Om_i)}^p
= \lim_{i\to\infty} C(p,q,R,N)[f]_{W^{1,p}(\Om\setminus\Om_i)}^p
=0.
\end{equation*}

Hence we can integrate \cref{pointwise} and interchange the limit and the integral to get the result.
\end{proof}

\begin{theorem}\label{Main2}
	Let $\Om\subset \RR^N$ be an open set. If $f\in L^p(\Om)\cap L^p(\Om)$ is such that 
	$$
	L_{p,q}(f):=\lim_{s\to1-} \I{\Om}\left((1-s)\I{B(x,\delta_x)}\frac{|f(x)-f(y)|^q}{|x-y|^{N+sq}}dy\right)^\frac{p}{q}dx<\infty,
	$$  
	then $f\in W^{1,p}(\Om)$ when $p>1$, and $f\in BV(\Om)$ when $p=1$.
\end{theorem}

\begin{proof}[\textbf{Proof of \Cref{Main2}}]
We divide the proof into two parts. First, we prove it for a particular case with a bit stronger assumptions, and then give the general proof.\\

\textbf{Step-1}: $\Om$ is bounded with $\Om\subseteq B(0,\lambda)$, and 
$$
\tilde{L}_{p,q}(f)
:=\lim_{s\to1-} \I{\Om}\left((1-s)\I{\Om}\frac{|f(x)-f(y)|^q}{|x-y|^{N+sq}}dy\right)^\frac{p}{q}dx
<\infty.
$$

 Extend $f$ by $0$, outside $\Om$. From the proof of Theorem~2 and 3 in \cite{BBM} we can see that for any $i=1,2,\cdots,N$, and $\varphi\in\test(\Om)$,
	\begin{equation}\label{eq-3}
		\left|\I{\Om}f(x)\frac{\partial \varphi(x)}{\partial x_i} dx\right|
		\leq C(\Om,N,p,q)(1-s)(J_{1,s}+J_{2,s}),
	\end{equation}
where, 
$$
J_{1,s}= \I{\Om}\I{\Om}\frac{|f(x)-f(y)|}{|x-y|^{1+N+sq-q}}|\varphi(y)|dydx
$$
and
$$
J_{2,s}=\I{\RR^N\setminus \Om}\I{\mbox{supp}\varphi}\frac{|f(y)||\varphi(y)|}{|x-y|^{1+N+sq-q}}dydx.
$$
We estimate $J_{1,s}$ using Fubuni's theorem to change the order of integration, then using H\"older's inequality twice, first with respect to the measure $\frac{dx}{|x-y|^{N+sq-q}}$ and then with respect to $dy$. We get
\begin{align*}
J_{1,s}
&\leq \I{\Om}\left(\I{\Om}\frac{|f(x)-f(y)|^q}{|x-y|^{q+N+sq-q}}dx\right)^\frac{1}{q} \left(\I{\Om}\frac{|\varphi(y)|^{q'}}{|x-y|^{N+sq-q}}dx\right)^\frac{1}{q'}dy\\
&\leq \left(\I{\Om}\left(\I{\Om}\frac{|f(x)-f(y)|^q}{|x-y|^{N+sq}}dx\right)^\frac{p}{q}dy\right)^\frac{1}{p} \left(\I{\Om}\left(\I{\Om}\frac{|\varphi(y)|^{q'}}{|x-y|^{N+sq-q}}dx\right)^\frac{p'}{q'}dy\right)^\frac{1}{p'}\\
&\leq \left(\I{\Om}\left(\I{\Om}\frac{|f(x)-f(y)|^q}{|x-y|^{N+sq}}dy\right)^\frac{p}{q}dx\right)^\frac{1}{p} \left(\I{\Om}|\varphi(y)|^{p'}\left(\I{B(0,\lambda)}\frac{dx}{|x-y|^{N+sq-q}}\right)^\frac{p'}{q'}dy\right)^\frac{1}{p'}\\
&= C(p,q,N,\lambda)(1-s)^\frac{-1}{q'} \left(\I{\Om}\left(\I{\Om}\frac{|f(x)-f(y)|^q}{|x-y|^{N+sq}}dy\right)^\frac{p}{q}dx\right)^\frac{1}{p} \|\varphi\|_{L^{p'}(\Om)}
\end{align*}
Now using the hypothesis of Step-1, we have
\begin{equation}\label{eq-2}
	(1-s)J_{1,s}
	\leq C(p,q,N,\lambda) \tilde{L}_{p,q}(f) \|\varphi\|_{L^{p'}(\Om)}.
\end{equation}

Using H\"older's inequality, we estimate $J_{2,s}$ as in \cite{BBM} to get

\begin{equation}\label{eq-1}
	(1-s)J_{2,s}
\leq C(N,p,q,\lambda) \|\varphi\|_{L^{p'}(\Om)} \|f\|_{L^{p}(\Om)}
\end{equation}

Using \cref{eq-1,eq-2,eq-3}, we get
\begin{equation*}
	\left|\I{\Om}f(x)\frac{\partial \varphi(x)}{\partial x_i} \right|
	\leq C(\Om,N,p,q,\lambda,f)\|\varphi\|_{L^{p'}(\Om)},
\end{equation*}
Hence by \cref{brezis-lemma}, the result follows.\\

\textbf{Step-2:} We now prove the theorem in full generality.

 For $1< p<\infty$, define $X^{1,p}(\Om):=W^{1,p}(\Om)$, and $X^{1,1}(\Om):=BV(\Om)$. Using \cref{exhaustion}, choose an increasing sequence of bounded open sets $\{\Om_n\}_n$ with smooth boundary such that $\cup_n\Om_n=\Om$, and $\mbox{dist}(x,\partial \Om)>\frac{1}{n}$, for $x\in\Om_n$. From the hypothesis, it follows that
 $$
 \lim_{s\to1-} \I{x\in\Om_n}\left((1-s)\I{y\in\Om_n\cap B(x,\delta_x)}\frac{|f(x)-f(y)|^q}{|x-y|^{N+sq}}dy\right)^\frac{p}{q}dx<\infty.
 $$
 We also have, for $s>\frac{1}{2}$, and $R>\frac{1}{n}$,
 \begin{align*}
 	\I{\Om_n}\left(\I{\Om_n\setminus B(x,\delta_x)}\frac{|f(x)-f(y)|^q}{|x-y|^{N+sq}}dy\right)^\frac{p}{q}dx
 	&\leq\I{\Om_n}\left(\I{\Om_n,\ |x-y|>\frac{1}{n}}\frac{|f(x)-f(y)|^q}{|x-y|^{N+sq}}dy\right)^\frac{p}{q}dx\\
 	&\leq n^{\frac{Np}{q}+sp}\I{\Om_n}\left(\I{\Om_n,\ |x-y|>\frac{1}{n}}|f(x)-f(y)|^qdy\right)^\frac{p}{q}dx\\
 	&\leq C(n,N,p,q) \left[ \|f\|_{L^p(\Om_n)}^p|\Om_n|^\frac{p}{q}
 	+ \|f\|_{L^q(\Om_n)}^p|\Om_n|\right].
 \end{align*}
Since, we have $f\in L^p(\Om)\cap L^p(\Om)$ frfom the hypotheses, and $\Om_n$ are bounded domains, we have
 $$
 \lim_{s\to1-} \I{\Om_n}\left((1-s)\I{\Om_n}\frac{|f(x)-f(y)|^q}{|x-y|^{N+sq}}dy\right)^\frac{p}{q}dx<\infty.
 $$
 From Step-1, we can conclude that $f\in X^{1,p}(\Om_n)$ for all $n$. Further the $X^{1,p}$-seminorms are uniformly bounded (independent of $n$) as can be seen from the following calculation, where we use \cref{Main1},
 \begin{align*}
 	K[f]_{X^{1,p}(\Om_n)}^p
 	&= \lim_{s\to1-} \I{x\in\Om_n}\left((1-s)\I{y\in B(x,\delta_{x,\Om_n})}\frac{|f(x)-f(y)|^p}{|x-y|^{N+sp}}dy\right)^\frac{p}{q}dx\\
 	&\leq \lim_{s\to1-} \I{x\in\Om}\left((1-s)\I{y\in B(x,\delta_x)}\frac{|f(x)-f(y)|^p}{|x-y|^{N+sp}}dy\right)^\frac{p}{q}dx\\
 	&=L_{p,q}(f)<\infty.
 \end{align*}
 The proof follows from the observation that $K[f]_{X^{1,p}(\Om)}^p
 =\sup_{n} K[f]_{X^{1,p}(\Om_n)}^p$.
\end{proof}

\bigskip

\section{Proof of \protect{\Cref{Main1*,Main2*}}}\label{Proof}
Note that \cref{Main2*} is a straightforward consequence of \cref{Main2}, as $L_{p,q}(f)\leq L^*_{p,q}(f)$. \Cref{Main1*} is also a consequence of \cref{Main1}, but it requires a bit more work. To complete the proof of \cref{Main1*}, we only need the following lemma:
\begin{lemma}\label{Equiv-Seminorm-cor}
	Let $\Om\subseteq\RR^N$ be an open set, $1\leq p <\infty$, $\tau\in(0,1)$, $f\in W^{1,p}(\Om)$ and $R>0$. Assume one of the following conditions
	\begin{enumerate}
		\item $1\leq q\leq \frac{Np}{N-p}$ with $p< N$, 
		\item $1\leq q<\infty$ with $p\geq N$.
	\end{enumerate} 
	Then \cref{bbm} implies \cref{bbm*}.
\end{lemma}
In order to prove this, we first prove a bit more general result.
\begin{proposition}\label{Equiv-Seminorm}
	Let $\Om\subseteq\RR^N$ be an open set, $1\leq p,q<\infty$, $\tau\in(0,1)$, $f\in L^p(\Om)\cap L^q(\Om)$ and $R>0$. Additionally, in the case $p<q$, assume that for some $s_0\in(0,1)$,
	\begin{equation}\label{eq7}
		\I{\Om}\left(\I{R\leq|h|\leq \tau \mbox{dist}(x,\partial \Om)}\frac{|f(x+h)-f(x)|^q}{|h|^{N+s_0q}}dh\right)^\frac{p}{q}dx<\infty.
	\end{equation}
Then \cref{bbm} implies \cref{bbm*}.
\end{proposition}
\begin{proof}[\textbf{Proof of \Cref{Equiv-Seminorm}}]
Note that, since 
$$
\I{\Om}\left(\I{B(x,\delta_x)}\frac{|f(x)-f(y)|^q}{|x-y|^{N+sq}}dy\right)^\frac{p}{q}dx
\leq 
\I{\Om}\left(\I{B(x,\tau\mbox{dist}(x,\partial \Om))}\frac{|f(x)-f(y)|^q}{|x-y|^{N+sq}}dy\right)^\frac{p}{q}dx,
$$
from \cref{bbm} we have
$$
\lim_{s\to1-} (1-s)^\frac{p}{q}\I{\Om}\left(\I{B(x,\tau\mbox{dist}(x,\partial \Om))}\frac{|f(x)-f(y)|^q}{|x-y|^{N+sq}}dy\right)^\frac{p}{q}dx
\geq K\I{\Om}|\nabla f(x)|^pdx.
$$
We focus on the reverse inequality. Observe that for $x,y\in\Om$, $\delta_x< |x-y|\leq \tau \mbox{dist}(x,\partial \Om)$ implies $R< |x-y|\leq \tau \mbox{dist}(x,\partial \Om)$. Hence we can write, using triangle inequality for $L^p$-norms,
\begin{align*}
&\left(\I{\Om}\left(\I{B(x,\tau \mbox{dist}(x,\partial \Om))}\frac{|f(x)-f(y)|^q}{|x-y|^{N+sq}}dy\right)^\frac{p}{q}dx\right)^\frac{1}{p} \\
	&= \left(\I{\Om}\left(\I{B(x,\delta_x)}\frac{|f(x)-f(y)|^q}{|x-y|^{N+sq}}dy
	+ \I{\delta_x\leq |x-y|\leq \tau \mbox{dist}(x,\partial \Om)}\frac{|f(x)-f(y)|^q}{|x-y|^{N+sq}}dy\right)^\frac{p}{q}dx\right)^\frac{1}{p} \\
	&\leq \left(\I{\Om}\left(\I{B(x,\delta_x)}\frac{|f(x)-f(y)|^q}{|x-y|^{N+sq}}dy\right)^\frac{p}{q}dx\right)^\frac{1}{p}
	+ \left(\I{\Om}\left(\I{R\leq |x-y|\leq \tau \mbox{dist}(x,\partial \Om)}\frac{|f(x)-f(y)|^q}{|x-y|^{N+sq}}dy\right)^\frac{p}{q}dx\right)^\frac{1}{p}\\
	&=:I_1^\frac{1}{p}+I_2^\frac{1}{p}.
\end{align*}
In order to complete the proof, in view of \cref{bbm}, we need to show that $I_2$ is bounded as $s\to 1-$.
We estimate $I_2$ in two separate cases.\\

\textbf{Case-1:} $1\leq q\leq p<\infty$.\\
Using Minkowsky's integral inequality and taking the $0$-extension of $f$ outside $\Om$, we have
\begin{align*}
	I_2^\frac{q}{p}
	&\leq 
	\left(\I{\Om}\left(\I{R\leq|h|\leq \tau \mbox{dist}(x,\partial \Om)}\frac{|f(x+h)-f(x)|^q}{|h|^{N+sq}}dh\right)^\frac{p}{q}dx\right)^\frac{q}{p}\\
	&\leq
	\I{\RR^N}\left(\I{\Om}\frac{|f(x+h)-f(x)|^p}{|h|^{\frac{Np}{q}+sp}}\chi_{B(0,\tau \mbox{dist}(x,\partial \Om))\setminus B(0,R)}(h)dx\right)^\frac{q}{p}dh\\
	&\leq
	\I{|h|\geq R}\frac{1}{|h|^{N+sq}} \left(\I{\Om}|f(x+h)-f(x)|^pdx\right)^\frac{q}{p}dh\\
	&\leq C(p,q) \I{|h|\geq R}\frac{1}{|h|^{N+sq}} \left(\I{\Om}|f(x)|^pdx\right)^\frac{q}{p}dh\\
	&\leq C(p,q,R,N) \|f\|_{L^p(\Om)}^q.
\end{align*}
Hence the proof follows in this case.\\

\textbf{Case-2:} $1\leq p\leq q<\infty$.\\
From \cref{eq7} we get that there is some $\lambda_f>0$ such that for $s\in (s_0,1)$,
\begin{align*}
	I_2
	&\leq \I{\Om}\left(\I{R\leq|h|\leq \tau \mbox{dist}(x,\partial \Om)}\frac{|f(x+h)-f(x)|^q}{|h|^{N+s_0q}}dh\right)^\frac{p}{q}dx\\
	&\leq 2\I{\Om\cap B(0,\lambda_f)}\left(\I{R\leq|h|\leq \tau \mbox{dist}(x,\partial \Om)}\frac{|f(x+h)-f(x)|^q}{|h|^{N+s_0q}}dh\right)^\frac{p}{q}dx.
\end{align*}
Using H\"older's inequality, we get
\begin{equation*}
	I_2
	\leq 2\lambda_f^{N(1-\frac{p}{q})}\left(\I{\Om\cap B(0,\lambda_f)}\I{R\leq|h|\leq \tau \mbox{dist}(x,\partial \Om)}\frac{|f(x+h)-f(x)|^q}{|h|^{N+s_0q}}dhdx\right)^\frac{p}{q}.
\end{equation*}
Now we can proceed as in Case-1 to show that $$I_2 \leq C(N,p,q,R,f) \|f\|_{L^q(\Om)}^p.$$
This completes the proof.
\end{proof}

Now we can prove \cref{Equiv-Seminorm-cor} and thereby complete the proof of \cref{Main1*}. 
\begin{proof}[\textbf{Proof of \cref{Equiv-Seminorm-cor}}]
By the standard embedding theorems, we already know that $f\in L^q(\Om)$. In order to prove the statement, we need to show that when $p<q$, \cref{eq7} holds. Let $\Om_1$ be a smooth domain such that 
$$
\{x\in\Om\ |\ \mbox{dist}(x,\partial\Om)>R\}
\subseteq \Om_1\subseteq \Om.
$$
Clearly $\Om_1$ is a $W^{1,p}$-extension domain (by \cref{stein}). Let $\tilde{f}\in W^{1,p}(\RR^N)$ be an extension of $f|_{\Om_1}$, that is 
$$
\|\tilde{f}\|_{W^{1,p}(\RR^N)}
\leq C(p,q,N,\Om) \|f\|_{W^{1,p}(\Om_1)}
\leq C(p,q,N,\Om) \|f\|_{W^{1,p}(\Om)}<\infty.
$$ 
This, along with \cref{triebel}
\begin{align*}
	\I{\Om}\left(\I{R\leq|h|\leq \tau \mbox{dist}(x,\partial \Om)}\frac{|f(x+h)-f(x)|^q}{|h|^{N+sq}}dh\right)^\frac{p}{q}dx
	&= \I{\Om_1}\left(\I{R\leq|h|\leq \tau \mbox{dist}(x,\partial \Om)}\frac{|f(x+h)-f(x)|^q}{|h|^{N+sq}}dh\right)^\frac{p}{q}dx\\
	&\leq \I{\RR^N}\left(\I{\RR^N}\frac{|\tilde{f}(x+h)-\tilde{f}(x)|^q}{|h|^{N+sq}}dh\right)^\frac{p}{q}dx\\
	&<\infty .
\end{align*}
\end{proof}

\bigskip

\section*{Acknowledgement} I extend my heartfelt gratitude to Antti V\"ah\"akangas for his invaluable discussions, meticulous review of the manuscript, and unwavering support throughout this project. I want to thank Pekka Koskela for the discussions we had during the manuscript preparation. Additionally, my appreciation goes to Emiel Lorist for pointing out an error in the initial version of the manuscript, and to the anonymous reviewer for his/her valuable comments and suggestions.\bigskip

\section*{Funding} The research is funded by Academy of Finland grant: Geometrinen Analyysi(21000046081).\bigskip

 \bibliography{bibliography.bib}

\begin{thebibliography}{10}

\bibitem{AlCiPi20}
Angela Alberico, Andrea Cianchi, Lubo\v{s} Pick, and Lenka Slav\'{\i}kov\'{a}.
\newblock On the limit as {$s \to 1^-$} of possibly non-separable fractional
  {O}rlicz-{S}obolev spaces.
\newblock {\em Atti Accad. Naz. Lincei Rend. Lincei Mat. Appl.},
  31(4):879--899, 2020.

\bibitem{AlCiPi}
Angela Alberico, Andrea Cianchi, Lubo\v{s} Pick, and Lenka Slav\'{\i}kov\'{a}.
\newblock Fractional {O}rlicz-{S}obolev embeddings.
\newblock {\em J. Math. Pures Appl. (9)}, 149:216--253, 2021.

\bibitem{AlCiPi21}
Angela Alberico, Andrea Cianchi, Lubo\v{s} Pick, and Lenka Slav\'{\i}kov\'{a}.
\newblock On fractional {O}rlicz-{S}obolev spaces.
\newblock {\em Anal. Math. Phys.}, 11(2):Paper No. 84, 21, 2021.

\bibitem{ArBo}
Adolfo Arroyo-Rabasa and Paolo Bonicatto.
\newblock A {B}ourgain-{B}rezis-{M}ironescu representation for functions with
  bounded deformation.
\newblock {\em Calc. Var. Partial Differential Equations}, 62(1):Paper No. 33,
  22, 2023.

\bibitem{BaMoRo}
Kaushik Bal, Kaushik Mohanta, and Prosenjit Roy.
\newblock Bourgain-{B}rezis-{M}ironescu domains.
\newblock {\em Nonlinear Anal.}, 199:111928, 10, 2020.

\bibitem{BaLi}
Sorina Barza and Martin Lind.
\newblock A new variational characterization of {S}obolev spaces.
\newblock {\em J. Geom. Anal.}, 25(4):2185--2195, 2015.

\bibitem{Bojarski}
B.~Bojarski.
\newblock Remarks on the {B}ourgain-{B}rezis-{M}ironescu approach to {S}obolev
  spaces.
\newblock {\em Bull. Pol. Acad. Sci. Math.}, 59(1):65--75, 2011.

\bibitem{BBM}
Jean Bourgain, Haim Brezis, and Petru Mironescu.
\newblock Another look at {S}obolev spaces.
\newblock In {\em Optimal control and partial differential equations}, pages
  439--455. IOS, Amsterdam, 2001.

\bibitem{BoBrMi02}
Jean Bourgain, Ha\"{\i}m Brezis, and Petru Mironescu.
\newblock Limiting embedding theorems for {$W^{s,p}$} when {$s\uparrow1$} and
  applications.
\newblock volume~87, pages 77--101. 2002.
\newblock Dedicated to the memory of Thomas H. Wolff.

\bibitem{BoNg}
Jean Bourgain and Hoai-Minh Nguyen.
\newblock A new characterization of {S}obolev spaces.
\newblock {\em C. R. Math. Acad. Sci. Paris}, 343(2):75--80, 2006.

\bibitem{BrGoDa}
Lorenzo Brasco, David G\'{o}mez-Castro, and Juan~Luis V\'{a}zquez.
\newblock Characterisation of homogeneous fractional {S}obolev spaces.
\newblock {\em Calc. Var. Partial Differential Equations}, 60(2):Paper No. 60,
  40, 2021.

\bibitem{BrScYu}
Denis Brazke, Armin Schikorra, and Po-Lam Yung.
\newblock Bourgain-{B}rezis-{M}ironescu convergence via {T}riebel-{L}izorkin
  spaces.
\newblock {\em Calc. Var. Partial Differential Equations}, 62(2):Paper No. 41,
  33, 2023.

\bibitem{Bre11}
Haim Brezis.
\newblock {\em Functional analysis, {S}obolev spaces and partial differential
  equations}.
\newblock Universitext. Springer, New York, 2011.

\bibitem{BrNg}
Ha\"{\i}m Brezis and Hoai-Minh Nguyen.
\newblock The {BBM} formula revisited.
\newblock {\em Atti Accad. Naz. Lincei Rend. Lincei Mat. Appl.},
  27(4):515--533, 2016.

\bibitem{BrSeVa}
Ha\"{\i}m Brezis, Andreas Seeger, Jean Van~Schaftingen, and Po-Lam Yung.
\newblock Sobolev spaces revisited.
\newblock {\em Atti Accad. Naz. Lincei Rend. Lincei Mat. Appl.},
  33(2):413--437, 2022.

\bibitem{BrVaYu}
Ha\"{\i}m Brezis, Jean Van~Schaftingen, and Po-Lam Yung.
\newblock Going to {L}orentz when fractional {S}obolev, {G}agliardo and
  {N}irenberg estimates fail.
\newblock {\em Calc. Var. Partial Differential Equations}, 60(4):Paper No. 129,
  12, 2021.

\bibitem{BrVaJe}
Ha\"{\i}m Brezis, Jean Van~Schaftingen, and Po-Lam Yung.
\newblock A surprising formula for {S}obolev norms.
\newblock {\em Proc. Natl. Acad. Sci. USA}, 118(8):Paper No. e2025254118, 6,
  2021.

\bibitem{Brezis02}
Kh. Brezis.
\newblock How to recognize constant functions. {A} connection with {S}obolev
  spaces.
\newblock {\em Uspekhi Mat. Nauk}, 57(4(346)):59--74, 2002.

\bibitem{BuGaTr}
Federico Buseghin, Nicola Garofalo, and Giulio Tralli.
\newblock On the limiting behaviour of some nonlocal seminorms: a new
  phenomenon.
\newblock {\em Ann. Sc. Norm. Super. Pisa Cl. Sci. (5)}, 23(2):837--875, 2022.

\bibitem{Calderon61}
A.-P. Calder\'{o}n.
\newblock Lebesgue spaces of differentiable functions and distributions.
\newblock In {\em Proc. {S}ympos. {P}ure {M}ath., {V}ol. {IV}}, pages 33--49.
  American Mathematical Society, Providence, R.I., 1961.

\bibitem{CuLaLu}
Xiaoyue Cui, Nguyen Lam, and Guozhen Lu.
\newblock Characterizations of second order {S}obolev spaces.
\newblock {\em Nonlinear Anal.}, 121:241--261, 2015.

\bibitem{DGPYYZ}
Feng Dai, Loukas Grafakos, Zhulei Pan, Dachun Yang, Wen Yuan, and Yangyang
  Zhang.
\newblock The {B}ourgain-{B}rezis-{M}ironescu formula on ball {B}anach function
  spaces.
\newblock {\em Math. Ann.}, 2023.

\bibitem{Dav}
J.~D\'{a}vila.
\newblock On an open question about functions of bounded variation.
\newblock {\em Calc. Var. Partial Differential Equations}, 15(4):519--527,
  2002.

\bibitem{DiSq}
Simone Di~Marino and Marco Squassina.
\newblock New characterizations of {S}obolev metric spaces.
\newblock {\em J. Funct. Anal.}, 276(6):1853--1874, 2019.

\bibitem{hhg}
Eleonora Di~Nezza, Giampiero Palatucci, and Enrico Valdinoci.
\newblock Hitchhiker's guide to the fractional {S}obolev spaces.
\newblock {\em Bull. Sci. Math.}, 136(5):521--573, 2012.

\bibitem{DoMi}
Oscar Dom\'{\i}nguez and Mario Milman.
\newblock Bourgain-{B}rezis-{M}ironescu-{M}az´ ya-{S}haposhnikova limit
  formulae for fractional {S}obolev spaces via interpolation and extrapolation.
\newblock {\em Calc. Var. Partial Differential Equations}, 62(2):Paper No. 43,
  37, 2023.

\bibitem{DrDu}
Irene Drelichman and Ricardo~G. Dur\'{a}n.
\newblock The {B}ourgain-{B}r\'{e}zis-{M}ironescu formula in arbitrary bounded
  domains.
\newblock {\em Proc. Amer. Math. Soc.}, 150(2):701--708, 2022.

\bibitem{FeSa}
Juli\'{a}n Fern\'{a}ndez~Bonder and Ariel~M. Salort.
\newblock Fractional order {O}rlicz-{S}obolev spaces.
\newblock {\em J. Funct. Anal.}, 277(2):333--367, 2019.

\bibitem{FeSa21}
Juli\'{a}n Fern\'{a}ndez~Bonder and Ariel~M. Salort.
\newblock Magnetic fractional order {O}rlicz-{S}obolev spaces.
\newblock {\em Studia Math.}, 259(1):1--24, 2021.

\bibitem{FeSq}
Gianluca Ferrari and Marco Squassina.
\newblock Nonlocal characterizations of variable exponent {S}obolev spaces.
\newblock {\em Asymptot. Anal.}, 127(1-2):121--142, 2022.

\bibitem{FeHa}
Rita Ferreira, Peter H\"{a}st\"{o}, and Ana~Margarida Ribeiro.
\newblock Characterization of generalized {O}rlicz spaces.
\newblock {\em Commun. Contemp. Math.}, 22(2):1850079, 25, 2020.

\bibitem{Foghem-thesis}
Guy Foghem.
\newblock L2-theory for nonlocal operators on domains.
\newblock {\em Publikationen an der Universit{\"a}t Bielefeld}, 2020.

\bibitem{Foghem}
Guy~Fabrice Foghem~Gounoue.
\newblock A remake of {B}ourgain-{B}rezis-{M}ironescu characterization of
  {S}obolev spaces.
\newblock {\em Partial Differ. Equ. Appl.}, 4(2):Paper No. 16, 36, 2023.

\bibitem{FoKaVo}
Guy~Fabrice Foghem~Gounoue, Moritz Kassmann, and Paul Voigt.
\newblock Mosco convergence of nonlocal to local quadratic forms.
\newblock {\em Nonlinear Anal.}, 193:111504, 22, 2020.

\bibitem{Gorny22}
Wojciech G\'{o}rny.
\newblock Bourgain-{B}rezis-{M}ironescu approach in metric spaces with
  {E}uclidean tangents.
\newblock {\em J. Geom. Anal.}, 32(4):Paper No. 128, 22, 2022.

\bibitem{Grafa04}
Loukas Grafakos.
\newblock {\em Classical and modern {F}ourier analysis}.
\newblock Pearson Education, Inc., Upper Saddle River, NJ, 2004.

\bibitem{GuHu}
Qingsong Gu and Qingzhong Huang.
\newblock Anisotropic versions of the {B}rezis--{V}an {S}chaftingen--{Y}ung
  approach at {$s=1$} and {$s=0$}.
\newblock {\em J. Math. Anal. Appl.}, 525(2):Paper No. 127156, 15, 2023.

\bibitem{Hajlas03}
Piotr Haj\l~asz.
\newblock A new characterization of the {S}obolev space.
\newblock volume 159, pages 263--275. 2003.
\newblock Dedicated to Professor Aleksander Pe\l czy\'{n}ski on the occasion of
  his 70th birthday (Polish).

\bibitem{HaRi}
Peter H\"{a}st\"{o} and Ana~Margarida Ribeiro.
\newblock Characterization of the variable exponent {S}obolev norm without
  derivatives.
\newblock {\em Commun. Contemp. Math.}, 19(3):1650022, 13, 2017.

\bibitem{Takashi}
Takashi Kamihigashi.
\newblock Interchanging a limit and an integral: necessary and sufficient
  conditions.
\newblock {\em J. Inequal. Appl.}, pages Paper No. 243, 9, 2020.

\bibitem{KoLe}
V.~I. Kolyada and A.~K. Lerner.
\newblock On limiting embeddings of {B}esov spaces.
\newblock {\em Studia Math.}, 171(1):1--13, 2005.

\bibitem{KrMo}
Andreas Kreuml and Olaf Mordhorst.
\newblock Fractional {S}obolev norms and {BV} functions on manifolds.
\newblock {\em Nonlinear Anal.}, 187:450--466, 2019.

\bibitem{LeSp}
Giovanni Leoni and Daniel Spector.
\newblock Characterization of {S}obolev and {$BV$} spaces.
\newblock {\em J. Funct. Anal.}, 261(10):2926--2958, 2011.

\bibitem{LeSp-corr}
Giovanni Leoni and Daniel Spector.
\newblock Corrigendum to ``{C}haracterization of {S}obolev and {$BV$} spaces''
  [{J}. {F}unct. {A}nal. 261 (10) (2011) 2926--2958].
\newblock {\em J. Funct. Anal.}, 266(2):1106--1114, 2014.

\bibitem{Ludwig14}
Monika Ludwig.
\newblock Anisotropic fractional {S}obolev norms.
\newblock {\em Adv. Math.}, 252:150--157, 2014.

\bibitem{Milman05}
Mario Milman.
\newblock Notes on limits of {S}obolev spaces and the continuity of
  interpolation scales.
\newblock {\em Trans. Amer. Math. Soc.}, 357(9):3425--3442, 2005.

\bibitem{Nguyen06}
Hoai-Minh Nguyen.
\newblock Some new characterizations of {S}obolev spaces.
\newblock {\em J. Funct. Anal.}, 237(2):689--720, 2006.

\bibitem{Nguyen08}
Hoai-Minh Nguyen.
\newblock Further characterizations of {S}obolev spaces.
\newblock {\em J. Eur. Math. Soc. (JEMS)}, 10(1):191--229, 2008.

\bibitem{NgPiSq}
Hoai-Minh Nguyen, Andrea Pinamonti, Marco Squassina, and Eugenio Vecchi.
\newblock New characterizations of magnetic {S}obolev spaces.
\newblock {\em Adv. Nonlinear Anal.}, 7(2):227--245, 2018.

\bibitem{NgPiSq20}
Hoai-Minh Nguyen, Andrea Pinamonti, Marco Squassina, and Eugenio Vecchi.
\newblock Some characterizations of magnetic {S}obolev spaces.
\newblock {\em Complex Var. Elliptic Equ.}, 65(7):1104--1114, 2020.

\bibitem{NgSq}
Hoai-Minh Nguyen and Marco Squassina.
\newblock On anisotropic {S}obolev spaces.
\newblock {\em Commun. Contemp. Math.}, 21(1):1850017, 13, 2019.

\bibitem{PiSqVe}
Andrea Pinamonti, Marco Squassina, and Eugenio Vecchi.
\newblock Magnetic {BV}-functions and the {B}ourgain-{B}rezis-{M}ironescu
  formula.
\newblock {\em Adv. Calc. Var.}, 12(3):225--252, 2019.

\bibitem{Ponce04}
Augusto~C. Ponce.
\newblock A new approach to {S}obolev spaces and connections to
  {$\Gamma$}-convergence.
\newblock {\em Calc. Var. Partial Differential Equations}, 19(3):229--255,
  2004.

\bibitem{PoSp}
Augusto~C. Ponce and Daniel Spector.
\newblock On formulae decoupling the total variation of {BV} functions.
\newblock {\em Nonlinear Anal.}, 154:241--257, 2017.

\bibitem{PrSa}
Mart\'{\i} Prats and Eero Saksman.
\newblock A {${\rm T}(1)$} theorem for fractional {S}obolev spaces on domains.
\newblock {\em J. Geom. Anal.}, 27(3):2490--2538, 2017.

\bibitem{RunSic}
Thomas Runst and Winfried Sickel.
\newblock {\em Sobolev spaces of fractional order, {N}emytskij operators, and
  nonlinear partial differential equations}, volume~3 of {\em De Gruyter Series
  in Nonlinear Analysis and Applications}.
\newblock Walter de Gruyter \& Co., Berlin, 1996.

\bibitem{SqVo}
Marco Squassina and Bruno Volzone.
\newblock Bourgain-{B}r\'{e}zis-{M}ironescu formula for magnetic operators.
\newblock {\em C. R. Math. Acad. Sci. Paris}, 354(8):825--831, 2016.

\bibitem{Stein70}
Elias~M. Stein.
\newblock {\em Singular integrals and differentiability properties of
  functions}.
\newblock Princeton Mathematical Series, No. 30. Princeton University Press,
  Princeton, N.J., 1970.

\bibitem{Triebel83}
Hans Triebel.
\newblock {\em Theory of function spaces}, volume~78 of {\em Monographs in
  Mathematics}.
\newblock Birkh\"{a}user Verlag, Basel, 1983.

\bibitem{Triebel11}
Hans Triebel.
\newblock Limits of {B}esov norms.
\newblock {\em Arch. Math. (Basel)}, 96(2):169--175, 2011.

\bibitem{YaYaYu}
Sibei Yang, Dachun Yang, and Wen Yuan.
\newblock New characterizations of {M}usielak-{O}rlicz-{S}obolev spaces via
  sharp ball averaging functions.
\newblock {\em Front. Math. China}, 14(1):177--201, 2019.

\bibitem{ZhYaYu}
Chenfeng Zhu, Dachun Yang, and Wen Yuan.
\newblock Extension theorem and bourgain--brezis--mironescu-type
  characterization of ball banach sobolev spaces on domains.
\newblock {\em ArXiv preprint arXiv: 2307.11392}, 2023.

\end{thebibliography}
\bibliographystyle{plain}
\end{document}